\documentclass[10pt,leqno]{amsart}
\usepackage{amssymb,amsthm}
\usepackage{amsmath}
\usepackage{setspace}
\usepackage{dcpic,pictexwd}
\usepackage[all]{xy}
\usepackage[pdftex]{graphicx}
\usepackage{mathrsfs}
\usepackage[pdftex]{hyperref}
\usepackage{bbm}

\theoremstyle{definition}
 \newtheorem{definition}{Definition}[section]

\theoremstyle{plain}

\theoremstyle{plain}
 \newtheorem{theorem}[definition]{Theorem}
 
\theoremstyle{definition}

\theoremstyle{plain}
 \newtheorem{lemma}[definition]{Lemma}

\theoremstyle{plain}
 \newtheorem{corollary}[definition]{Corollary}

\theoremstyle{remark}
 \newtheorem{remark}[definition]{Remark}

\theoremstyle{definition}

\theoremstyle{plain}

\newcommand{\Hom}{\mathrm{Hom}}

\newcommand{\Z}{\mathbb{Z}}

\newcommand{\A}{\Lambda}

\renewcommand{\k}{\Bbbk}
\renewcommand{\1}{\mathbbm{1}}

\newcommand{\M}{\widehat{M}}
\newcommand{\N}{\widehat{N}}
\newcommand{\Y}{\widehat{Y}}

\renewcommand{\P}{\widehat{P}}
\newcommand{\Ar}{\widehat{\Lambda}}
\newcommand{\h}{\widehat{h}}

\hypersetup{
    bookmarks=true,         
    unicode=false,          
    pdftoolbar=true,        
    pdfmenubar=true,        
    pdffitwindow=false,     
    pdfstartview={FitH},    
    pdfnewwindow=true,      
    colorlinks=true,       
    linkcolor=red,          
    citecolor=blue,        
    filecolor=magenta,      
    urlcolor=blue           
}

\title[On irreducible morphisms and Auslander-Reiten triangles]{On irreducible morphisms and Auslander-Reiten triangles in the stable category of modules over repetitive algebras} 
\thanks{This research was partly supported by the Centro de Ciencias de Matem‡ticas, UNAM, Unidad Morelia, (Mexico), CODI and Estrategia de Sostenibilidad 2019-2020 (Universidad de Antioquia, UdeA), and COLCIENCIAS (CONVOCATORIA DOCTORADOS NACIONALES No. 727 DE 2015).}
\author[Calder\'on-Henao]{Yohny Calder\'on-Henao}
\address{Instituto de Matem\'aticas, Universidad de Antioquia, Medell{\'\i}n, Antioquia, Colombia}
\email{yohny.calderon@udea.edu.co}
\author[Giraldo]{Hern\'an Giraldo}
\address{Instituto de Matem\'aticas, Universidad de Antioquia, Medell{\'\i}n, Antioquia, Colombia}
\email{hernan.giraldo@udea.edu.co}
\author[V\'elez-Marulanda]{Jos\'e A. V\'elez-Marulanda}
\address{Department of Mathematics, Valdosta State University,
2072 Nevins Hall, 1500 N. Patterson St, Valdosta, GA,  31698-0040}
\email{javelezmarulanda@valdosta.edu (Corresponding author)}

\begin{document}
\renewcommand{\labelenumi}{\textup{(\roman{enumi})}}
\renewcommand{\labelenumii}{\textup{(\roman{enumi}.\alph{enumii})}}
\numberwithin{equation}{section}

\begin{abstract}

Let $\k$ be an algebraically closed field, let $\A$ be a finite dimensional $\k$-algebra, and let $\Ar$ be the repetitive algebra of $\A$.  For the stable category of finitely generated left $\Ar$-modules $\Ar$-\underline{mod}, we show that the irreducible morphisms fall into three canonical forms: (i) all the component morphisms are split monomorphisms; (ii) all of them are split epimorphisms; (iii) there is exactly one irreducible component. We next use this fact in order to describe the shape of the Auslander-Reiten triangles in $\Ar$-\underline{mod}. We use the fact (and prove) that every Auslander-Reiten triangle in $\Ar$-\underline{mod} is induced from an Auslander-Reiten sequence of finitely generated left $\Ar$-modules. 
\end{abstract}
\subjclass[2010]{16G10 \and 16G20 \and 20C20}
\maketitle
\section{Introduction}
Let $\k$ be an algebraically closed field and $\A$ be a finite dimensional $\k$-algebra. Denote by $\A$-mod the abelian category of finitely generated left $\A$-modules and by $\mathcal{D}^b(\A\textup{-mod})$ the bounded derived category of $\A$, which is a triangulated category. Let $\Ar$ be the repetitive algebra of $\A$ (see \S \ref{repetitivealgebra}) and denote by $\Ar$-mod the abelian category of finitely generated left $\Ar$-modules and by $\Ar$-\underline{mod} its stable category. It follows from \cite[Chap. II, \S 2.2]{happel} that $\Ar$-\underline{mod} is a triangulated category. Due to a fundamental result due to D. Happel, there exists a full and faithful exact functor of triangulated categories  $\mu: \mathcal{D}^b(\A\textup{-mod})\to \Ar\textup{-\underline{mod}}$, such that $\mu$ extends the identity functor on $\A$-mod, where $\A$-mod is embedded in $\mathcal{D}^b(\A\textup{-mod})$ (resp. $\Ar$-\underline{mod}) as complexes (resp. modules) concentrated in degree zero. Moreover, $\mu$ is an equivalence if and only if $\A$ has finite global dimension (see \cite[\S 2.3]{happel2}). This result allows us to describe the indecomposable objects and the corresponding morphisms in $\mathcal{D}^b(\A\textup{-mod})$ via $\Ar$. However, it is also well-known that the description of this functor $\mu$ is rather difficult to visualize (see \cite{barot-mendoza}). In \cite{happkellrei}, D. Happel et al. investigated the relationship between $\mathcal{D}^b(\A\textup{-mod})$ and $\Ar$-\underline{mod} from various points of view for algebras of infinite global dimension.  One the other hand, it was proved by D. Hughes and  J. Waschb\"usch in \cite[\S 2.5]{hughes-wash} that $\Ar$-mod has Auslander-Reiten sequences and by D. Happel in \cite{happel} that $\Ar$-\underline{mod} has Auslander-Reiten triangles. Therefore, it is a natural task to investigate the behavior of Auslander-Reiten sequences and triangles involving left $\Ar$-modules, which in turn rises the task of investigating the behavior of irreducible morphisms between left $\Ar$-modules. Such task was approached in the more general setting of additive categories by M. J. Souto Salorio and R. Bautista in \cite[\S 2]{bautista1}. More recently, the second author investigated in \cite{giraldo} the behavior of irreducible morphisms between objects in $\Ar$-mod. This work was inspired by the second author's joint work with H. Merklen on the behavior of irreducible morphisms between objects in $\mathcal{D}^b(\A\textup{-mod})$ (see \cite{giraldo-merklen}). This approach has been used recently by the second author together with E. Ribeiro Alvares and S. M. Fernandes in \cite{ribeiro-fernandes-giraldo} to investigate the shape of Auslander-Reiten triangles in $\mathcal{D}^b(\A\textup{-mod})$. The goal of this article is to investigate the shape of Auslander-Reiten triangles in $\Ar$-\underline{mod} and thus recover the results in \cite{ribeiro-fernandes-giraldo}. Our main motivation is that the objects and the morphisms in $\Ar$-mod are easier to understand than those in 
$\mathcal{D}^b(\A\textup{-mod})$. Our main results are Theorem \ref{thm1}, which extends \cite[Thm. 26]{giraldo} to $\Ar$-\underline{mod} and Theorem \ref{thm1}, which provides a generalization of \cite[\S 4.1]{ribeiro-fernandes-giraldo}. 

In Theorem \ref{thm1}, we prove that an irreducible morphism in $\Ar$-\underline{mod} falls into three possible canonical forms: (i) all the component morphisms are split monomorphisms; (ii) all of them are split epimorphisms; (iii) there is exactly one irreducible component. On the other hand, in Theorem \ref{thm2} we describe the shape of the irreducible morphisms involved in Auslander-Reiten triangles with terms in $\Ar$-\underline{mod}. In order to achive our goals, we use the fact that every Auslander-Reiten triangle in $\Ar$-\underline{mod} is induced by an Auslander-Reiten sequence in $\Ar$-mod (see Theorem \ref{thm0}). 

This article is organized as follows. In \S \ref{sec2}, we review the definition of $\Ar$ and some basic aspects concerning the categories $\Ar$-mod and $\Ar$-\underline{mod}. We also review the results from \cite{giraldo} concerning irreducible morphisms in $\Ar$-mod as well as the definition of Auslander-Reiten sequences in $\Ar$-mod and Auslander-Reiten triangles in $\Ar$-\underline{mod}. We also prove Theorem \ref{thm0} and some useful corollaries.  In \S \ref{sec3}, we prove Theorem \ref{thm2}. Finally, in \S \ref{sec4}, we provide an example of a finite dimensional gentle $\k$-algebra $\A$ of infinite global dimension that verifies the results in Theorem \ref{thm2}. 

This article constitutes the doctoral dissertation of the first author under the supervision of the other two.

\section{Preliminary results}\label{sec2}
Throughout this article, we assume that $\k$ is a fixed algebraically closed field of arbitrary characteristic. If $f:A\to B$ and $g:B\to C$ are morphisms in a given category $\mathcal{C}$, then the composition of $f$ with $g$ is denoted by $g\circ f$, i.e., the usual composition of morphisms. All modules considered in this article are assumed to be finitely generated, and unless explicitly stated, they will be from the left side. Let $\A$ be a fixed finite dimensional basic $\k$-algebra.  We denote by $\A$-mod the abelian category of finitely generated $\A$-modules, and by $D(-)=\Hom_\k(-,\k)$ the standard $\k$-duality on $\A$-mod. In particular $Q=D\A$ is the minimal injective cogenerator of $\A$-mod. Note that $Q$ is also a finitely generated $\A$-$\A$-bimodule in the following way. For all $\varphi \in Q$, and $a', a''\in \A$, $a\varphi a'$ is the $\k$-linear morphism that sends each $a\in\A$ to $\varphi(a'aa'')$.

\subsection{Repetitive algebras}\label{repetitivealgebra}

Although the repetitive algebra of a finite dimensional $\k$-algebra was originally introduced by  D. Hughes and J. Waschb\"usch in \cite{hughes-wash}, in this article, we follow the notation in \cite[Chap. II, \S 2]{happel}.

In the following, we recall that definition of the {\it repetitive algebra} $\Ar$ of $\A$.
\begin{enumerate}
\item The underlying $\k$-vector space of $\Ar$  is given by $\Ar=\left(\bigoplus_{i\in \Z} \A\right)\oplus \left(\bigoplus_{i\in \Z} Q\right)$.
\item The elements of $\Ar$ are denoted by $(a_i, \varphi_i)_i$, where $a_i\in \A$, $\varphi_i\in Q$ and almost all of the $a_i, \varphi_i$ are zero.   
\item The product of two elements $(a_i, \varphi_i)_i$ and $(b_i,\psi_i)_i$ in $\Ar$ is defined as 
\[(a_i,\varphi_i)_i\cdot (b_i,\psi_i)_i= (a_ib_i, a_{i+1}\psi_i+\varphi_ib_i)_i.\]
\end{enumerate}
We refer the reader to \cite[Chap. II, \S 2]{happel} for an interpretation of $\Ar$ as a doubly infinite matrix $\k$-algebra. 
The finitely generated $\Ar$-modules are of the form $\M=(M_i,f_i)_{i\in \Z}$, where for each $i\in \Z$, $M_i$ is a finitely generated $\A$-module, all but finitely many being zero, and $f_i$ is the morphism of $\A$-modules $f_i:Q\otimes_\A M_i\to M_{i+1}$ such that $f_{i+1}\circ(\mathrm{id}_Q\otimes f_i)=0$, where $\mathrm{id}_Q$ denotes the identity morphism on $Q$.   We can visualize a finitely generated $\Ar$-module $\M$ as follows:
\begin{equation}\label{tensorcomplx}
\xymatrix@=20pt{
\cdots\ar@{~>}[r]&M_{i-2}\ar@{~>}[r]^{f_{i-2}}&M_{i-1}\ar@{~>}[r]^{f_{i-1}}&M_i\ar@{~>}[r]^{f_i}&M_{i+1}\ar@{~>}[r]^{f_{i+1}}&M_{i+2}\ar@{~>}[r]&\cdots
}
\end{equation} 

A morphism $\h:\M\to \M'$ of finitely generated $\Ar$-modules is a sequence $\h=(h_i)_{i\in \Z}$ of morphisms of $\A$-modules $h_i: M_i\to M_i'$ such that, for all $i\in \Z$, $h_{i+1}\circ f_i=f_i'\circ (\mathrm{id}_Q\otimes h_i)$.
We denote by $\Ar$-mod the abelian category of finitely generated $\Ar$-modules. By using the description of objects in $\Ar$-mod given in (\ref{tensorcomplx}), we can also visualize morphisms $\h:\M\to \M'$ in $\Ar$-mod as follows.

\begin{equation}\label{tensorcomplxmorph}
\xymatrix@=20pt{
\M\ar[d]_{\h}:&\cdots\ar@{~>}[r]&M_{i-2}\ar@{~>}[r]^{f_{i-2}}\ar[d]^{h_{i-2}}&M_{i-1}\ar@{~>}[r]^{f_{i-1}}\ar[d]^{h_{i-1}}&M_i\ar@{~>}[r]^{f_i}\ar[d]^{h_i}&M_{i+1}\ar@{~>}[r]^{f_{i+1}}\ar[d]^{h_{i+1}}&M_{i+2}\ar@{~>}[r]\ar[d]^{h_{i+2}}&\cdots\\
\M':&\cdots\ar@{~>}[r]&M'_{i-2}\ar@{~>}[r]^{f'_{i-2}}&M'_{i-1}\ar@{~>}[r]^{f'_{i-1}}&M'_i\ar@{~>}[r]^{f'_i}&M'_{i+1}\ar@{~>}[r]^{f'_{i+1}}&M'_{i+2}\ar@{~>}[r]&\cdots
}
\end{equation} 

In the situation of (\ref{tensorcomplxmorph}), we call each $h_i$ the {\it $i$-th component morphism} of $\h$. 

\begin{remark}\label{rem1.1}
Let $T\A$ be the {\it trivial extension} of $\A$, i.e., $T\A=\A\oplus Q$ as $\k$-vector spaces and the product is defined as 
\begin{equation*}
(a,\varphi)\cdot (b, \psi)=(ab,a\psi+\varphi b), 
\end{equation*} 
for all $a,b\in \A$ and $\varphi,\psi\in Q$.

It was noted in \cite[Chap. II, \S2.4]{happel} that $T\A$ is a $\Z$-graded algebra, where the elements of $\A\oplus Q$ are the elements of degree $0$ and those in $0\oplus Q$ are the elements of degree $1$.  Moreover, if $T\A^{\Z}$-mod is the category of finitely generated $\Z$-graded $T\A$-modules with morphisms of degree $0$, then the categories $\Ar$-mod and $T\A^{\Z}$-mod are equivalent. Moreover, by the remarks in \cite[Chap. VI, \S 2]{maclane}, it follows that if $\M=(M_i,f_i)_{i\in \Z}$, $\M'=(M'_i,f'_i)_{i\in \Z}$ and $\M''=(M''_i,f''_i)_{i\in \Z}$ are $\Ar$-modules, then 
$0\to \M\xrightarrow{\h} \M'\xrightarrow{\h'} \M''\to 0$ is exact in $\Ar$-mod if and only if $0\to M_i\xrightarrow{h_i}M'_i\xrightarrow{h'_i}M''_i\to 0$ is exact in $\A$-mod for all $i\in \Z$.
\end{remark}

It follows from \cite[Chap II, \S 2.2]{happel} that $\Ar$-mod is a {\it Frobenius category} in the sense of \cite[Chap. I, \S 2]{happel}, i.e., $\Ar$-mod is an exact category in the sense of \cite{quillen} which has enough projective as well as injective objects, and these classes of objects coincide. As a matter of fact, the indecomposable projective-injective $\Ar$-modules are given by
\begin{equation}\label{projindec}
\xymatrix@=20pt{
\cdots\ar@{~>}[r]&0\ar@{~>}[r]&\Hom_\A(Q,I)\ar@{~>}[r]^{\hspace*{25pt}\varphi_I}&I\ar@{~>}[r]&0\ar@{~>}[r]&\cdots,
}
\end{equation}
where $I$ is an indecomposable injective $\A$-module and $\varphi_I$ is the isomorphism of $\A$-modules $Q\otimes_\A\Hom_\A(Q,I)\to I$. We denote by $\Ar$-\underline{mod} the stable category of $\Ar$-mod, i.e., the objects of $\Ar$-\underline{mod} are the same as those in $\Ar$-mod, and two morphisms $\underline{\h}, \underline{\h}': \M\to \M'$ in $\Ar$-\underline{mod} are identified provided that $\h-\h'$ factors through a finitely generated projective-injective $\Ar$-module. For all objects $\M$ in $\Ar$-mod, we denote by $\Omega^{-1}\M$ the first cozyzygy of $\M$, i.e., $\Omega^{-1}\M$ is the cokernel of an injective $\Ar$-module hull $\M\to I(\M)$, which is unique up to isomorphism. It follows from \cite[Chap. I, \S 2.2]{happel} that $\Omega^{-1}$ induces an automorphism of $\Ar$-\underline{mod}. Moreover, it follows from \cite[Chap. I, \S 2.6]{happel} that $\Ar$-\underline{mod} is a triangulated category in the sense of \cite{verdier} whose translation functor is $\Omega^{-1}$. More precisely, a distinguished triangle 

\begin{equation}\label{triangle}
\M\xrightarrow{\underline{\h}}\M'\xrightarrow{\underline{\h}'}\M''\xrightarrow{\underline{\h}''} \Omega^{-1}\M
\end{equation}
in $\Ar$-\underline{mod} comes from a pushout diagram of $\Ar$-modules
\begin{equation*}
\xymatrix@=20pt{
0\ar[r]&\M\ar[r]\ar[d]_{\h}&I(\M)\ar[r]\ar[d]&\Omega^{-1}\M\ar[r]\ar[d]&0\\
0\ar[r]&\M'\ar[r]&\M''\ar[r]&\Omega^{-1}\M\ar[r]&0
},
\end{equation*}
where $\underline{\h}$ is the stable class of $\h$ in $\Ar$-\underline{mod}. Moreover, every short exact sequence in $\Ar$-mod induces a distinguished triangle in $\Ar$-\underline{mod} (see also \cite[Lemma 1.2]{chenzhang} for more details). 
Let $\mathcal{D}^b(\A\textup{-mod})$ be the bounded derived category of $\A$, which is also a triangulated category (see \cite{verdier}). It follows from a fundamental result due to D. Happel (see \cite[\S 2.3]{happel2}) that $\mathcal{D}^b(\A\textup{-mod})$ and $\Ar$-\underline{mod} are equivalent as triangulated categories if and only if  $\A$ has finite global dimension. 

\subsection{Irreducible morphisms in $\Ar$-mod}

Let $\mathcal{A}$ be an arbitrary additive category.  Recall that a morphism $f:X\to Y$ in $\mathcal{A}$ a said to be a {\it split monomorphism} (resp. {\it split epimorphism}) provided that there exists $g:Y\to X$ in $\mathcal{A}$ such that  $g\circ f=\mathrm{id}_X$ (resp. $f\circ g=\mathrm{id}_Y$). Throughout the remainder of this article, we refer to a split monomorphism (resp. split epimorphism) as a split mono (resp. split epi).  Recall also that $f$ is said to be {\it irreducible} if $f$ is neither a split mono nor a split epi and if $f=v\circ u$ for some morphisms $u:X\to Z$ and $v:Z\to Y$ in $\mathcal{A}$, then either $u$ is a split mono or $v$ is a split epi.  
\begin{remark}
If $\M$ and $\M'$ are indecomposable $\Ar$-modules, then $\h: \M\to \M'$ is irreducible in $\Ar$-mod if and only if $\h\in \mathrm{rad}_{\Ar}(\M,\M')/\mathrm{rad}_{\Ar}^2(\M,\M')$, where $\mathrm{rad}_{\Ar}(\M,\M')$ is the $\k$-vector space of morphisms from $\M$ to $\M'$ that are not invertible and $\mathrm{rad}_{\Ar}^2(\M,\M')$ is the $\k$-vector space of morphisms of the form $\h'\circ \h''$ such that $\h'':\M\to\M''\in \mathrm{rad}_{\Ar}(\M,\M'')$ and $\h':\M''\to\M'\in \mathrm{rad}_{\Ar}(\M'',\M')$ for some $\Ar$-module $\M''$. In this situation, we put 
\[\mathrm{Irr}(\M,\M')= \mathrm{rad}_{\Ar}(\M,\M')/\mathrm{rad}_{\Ar}^2(\M,\M').\]
Note in particular that $\mathrm{Irr}(\M,\M')$ is also a $\k$-vector space. On the other hand, if $\M=\bigoplus_{j=1}^n\M_j$, and $\M'=\bigoplus_{k=1}^m\M'_k$ are direct sums of indecomposable $\Ar$-modules, then $\mathrm{rad}_{\Ar}(\M,\M')$ is the $\k$-vector space of morphisms $\M\to \M'$ such that no component morphism $\M_j\to \M'_k$ is an isomorphism. For more details, we refer the reader to \cite[\S 2.2]{ringel2} and \cite[App. A.3, \S 3.4]{assem3}.  

\end{remark}

\begin{lemma}\label{lemma1.2}
Assume that $\h:\M\to\M'$ is an irreducible morphism in  $\Ar$-\textup{mod}.
\begin{enumerate}
\item If $\M$ is indecomposable and $\M'=\bigoplus_{k=1}^m\M'_k$ is a direct sum of indecomposable $\Ar$-modules, then for all $1\leq k\leq m$, the induced morphism $\h_k:\M\to \M'_k$ of $\Ar$-modules is irreducible.
\item If $\M=\bigoplus_{i=1}^n\M_i$ is a direct sum of indecomposable $\Ar$-modules and $\M'$ is indecomposable, then for all $1\leq i\leq m$, the induced morphism $\h_i:\M_i\to \M'$ of $\Ar$-modules is irreducible.
\end{enumerate}
\end{lemma}
\begin{proof}
It is enough to prove (i) for the statement in (ii) is obtained in a similar way. Assume then that $\h=(\h_1,\ldots,\h_m)^t: \M\to \bigoplus_{k=1}^m\M'_k$ is irreducible in $\Ar$-mod. Then it follows that for all $1\leq k\leq m$, $\h_k$ is neither a split mono nor a split epi. This in particular implies that $\h\in \mathrm{rad}_{\Ar}(\M,\M')$. Let $k\in \{1,\ldots,m\}$ be fixed but arbitrary, and consider the natural projection $\widehat{\pi}_k: \M'\to \M'_k$ which is a split epi. It follows from \cite[Prop. 2.18]{bautista1} that $\h_k=\widehat{\pi}_k\circ \h$ is also an irreducible morphism in $\Ar$-mod. This finishes the proof of Lemma \ref{lemma1.2}. 
\end{proof}

\begin{theorem}{\textup{(\cite[Prop. 41 \& Thm. 42]{giraldo})}}\label{giraldo0}
Let $\h:\M\to \M'$ be a morphism in $\Ar$-\textup{mod} such that neither $\M$ nor $\M'$ has projective direct summands, and as before, denote by $\underline{\h}$ its corresponding class in $\Ar$-\textup{\underline{mod}}. Then $\h$ is split mono (resp. split epic, resp. irreducible) if and only if $\underline{\h}$ is split mono (resp. split epic, resp. irreducible). 
\end{theorem}

We need the following consequence of Theorem \ref{giraldo0}. 

\begin{corollary}\label{cor0}
Let $\h,\h':\M\to \M'$ be morphisms in $\Ar$-\textup{mod} with either $\M$ or $\M'$ indecomposable and both without  projective direct summands. If $\underline{\h}=\underline{\h}'$ in $\Ar$-\underline{\textup{mod}} and $\underline{\h}$ is further irreducible, then $\h=\h'$ in $\Ar$-\textup{mod}.  
\end{corollary}

\begin{proof}
Assume without loss of generality that $\M$ is indecomposable and that $\M'=\bigoplus_{k=1}^m\M'_k$ is a direct sum of indecomposable $\Ar$-modules. Thus we can also assume that $\h=(\h_{k})$ and $\h'=(\h'_{k})$, where for each $1\leq k\leq m$,  $\h_{k},\, \h'_{k}: \M\to \M'_k$ is a morphism of $\Ar$-modules. Note that by Theorem \ref{giraldo0}, since $\underline{\h}=\underline{\h'}$ in $\Ar$-\underline{mod}, it follows that $\h$ and $\h'$ are both irreducible morphisms in $\Ar$-mod. Moreover, $\h-\h'=\widehat{v}\circ \widehat{u}$, where $\widehat{u}:\M\to \P$ and $\widehat{v}:\P\to \M'$, where $\P$ is a projective $\Ar$-module. Thus for each $1\leq k\leq m$, there exists an indecomposable projective $\Ar$-module $\P_{k}$ such that $\h_{k}-\h'_{k}=\widehat{v}_{k}\circ \widehat{u}_{k}$ where $\widehat{u}_{k}:\M\to \P_{k}$ and $\widehat{v}_{k}:\P_{k}\to \M'_{k}$. Assume that $\h\not=\h'$ in $\Ar$-mod. It follows that there exists $k_0\in \{1,\ldots,m\}$ such that $\h_{k_0}-\h'_{k_0}=\widehat{v}_{k_0}\circ \widehat{u}_{k_0}\not=0$. Since $\h$ and $\h'$ are both irreducible in $\Ar$-mod, it follows by Lemma \ref{lemma1.2} (i) that $\h_{k_0}$ and $\h'_{k_0}$ are irreducible, i.e. $\h_{k_0}, \h'_{k_0}\in \mathrm{Irr}(\M,\M'_{k_0})$. Since the latter is a $\k$-vector space, it follows that $\widehat{v}_{k_0}\circ \widehat{u}_{k_0}\in \mathrm{Irr}(\M,\M'_{k_0})$, i.e. $\widehat{v}_{k_0}\circ \widehat{u}_{k_0}$ is irreducible in $\Ar$-mod. Therefore, $\widehat{v}_{k_0}$ is split epic or $\widehat{u}_{k_0}$ is split mono, which in turn gives that either of the $\Ar$-modules $\M$ and $\M'_{k_0}$ is a direct summand of $\P_{k_0}$ and thus either of $\M$ or $\M'$ has projective direct summands. This contradicts our assumption, and hence $\h=\h'$ in $\Ar$-mod.

\end{proof}

Let $\h: \M\to \M'$ be a morphism in $\Ar$-mod. Following \cite[Def. 14]{giraldo}),  we say that $\h$ is {\it smonic} (resp. {\it sepic}) if in the situation of (\ref{tensorcomplxmorph}), $h_i$ is split mono (resp. split epi) for all $i\in \Z$. 


The following result classifies the shape of the irreducible morphisms in $\Ar$-mod. 

\begin{theorem}{\textup{(\cite[Thm. 26]{giraldo})}}\label{shapeirred}
Let $\h:\M\to \M'$ be an irreducible morphism in $\Ar$-\textup{mod}. Then under the situation of (\ref{tensorcomplxmorph}), $\h$ satisfies one of the following conditions.
\begin{enumerate}
\item For all $i\in \Z$, $h_i$ is a split mono, i.e., $\h$ is smonic.
\item For all $i\in \Z$, $h_i$ is a split epi, i.e., $\h$ is sepic.
\item There exists $i_0\in \Z$ such that $h_{i_0}$ is neither a split mono nor a split epi. In this situation, $i_0$ is unique and $h_{i_0}$ is an irreducible morphism in $\A$-\textup{mod}.   
\end{enumerate}
\end{theorem}

An irreducible morphism $\h:\M\to \M'$ in $\Ar$-mod satisfying the situation in Theorem \ref{shapeirred} (iii) is said to be {\it sirreducible}.

\subsection{Auslander-Reiten sequences in $\Ar$-mod and Auslander-Reiten triangles in $\Ar$-\underline{mod}}

Following \cite[\S 2.3]{ringel2}, a short exact sequence in $\Ar$-mod

\begin{equation}\label{almostsplitseq}
0\to \M\xrightarrow{\h}\M'\xrightarrow{\h'}\M''\to 0
\end{equation}

is an {\it Auslander-Reiten sequence} if the following properties are satisfied.

\begin{enumerate}
\item[(ARS1)] The morphism $\h$ is not split mono and for all morphisms $\widehat{v}: \M\to \N$ in $\Ar$-mod that are not split mono, there exists a morphism $\widehat{u}:\M'\to \N$ in $\Ar$-mod such that $\widehat{v}=\widehat{u}\circ \h$, i.e., $\h$ is left almost split. 
\item[(ARS2)] The morphism $\h'$ is not split epic and for all morphisms $\widehat{u}': \N'\to \M''$ in $\Ar$-mod that are not split epi, there exists a morphism $\widehat{v}':\N'\to \M'$ in $\Ar$-mod such that $\widehat{u}'=\h'\circ \widehat{v}'$, i.e, $\h'$ is right almost split. 
\end{enumerate} 

Assume that (\ref{almostsplitseq}) is an Auslander-Reiten sequence in $\Ar$-mod. In this situation, we say that (\ref{almostsplitseq}) starts (resp. ends ) at $\M$ (resp. $\M''$). By \cite[\S 2.5]{hughes-wash}, it follows that $\Ar$-mod has Auslander-Reiten sequences, i.e. for all indecomposable $\Ar$-modules $\M$ (resp. $\M''$) in $\Ar$-mod, there is an Auslander-Reiten sequence that starts (resp. ends) at $\M$ (resp. $\M''$). 

Following \cite[Chap. I, \S 4]{happel}, we say that a distinguished triangle as in (\ref{triangle}) in $\Ar$-\underline{mod} is an {\it Auslander-Reiten triangle} if the following properties are satisfied.
\begin{enumerate}
\item[(ART1)] The $\Ar$-modules $\M$ and $\M''$ are indecomposable.
\item[(ART2)] The morphism $\underline{\h}'': \M''\to \Omega^{-1}\M$ is non-zero.
\item[(ART3)] If $\underline{\widehat{u}}':\N'\to \M''$ is a morphism in $\Ar$-\underline{mod} which is not split epic, then there exists $\underline{\widehat{v}}':\N'\to \M'$ such that $\underline{\widehat{u}}'=\underline{\h}'\circ \underline{\widehat{v}}'$.   
\end{enumerate}  

Assume that (\ref{triangle}) is an Auslander-Reiten triangle. It follows from \cite[\S 1.4]{happel} that $\underline{\h}$ is not a split mono and that  $\underline{\h}'$ is not a split epic. Moreover, by \cite[Chap. I, \S 4.2]{happel}, it follows that the following condition is satisfied.
\begin{enumerate}
\item[(ART3$^\ast$)] If $\underline{\widehat{v}}: \M\to \N$ is a morphism in $\Ar$-\underline{mod} which is not split mono, then there exists $\underline{\widehat{u}}:\M'\to \N$ such that $\underline{\widehat{v}}=\underline{\widehat{u}}\circ \underline{\h}$.
\end{enumerate}

As before, if (\ref{triangle}) is an Auslander-Reiten triangle, then we say that (\ref{triangle}) starts (resp. ends) at $\M$ (resp. $\M''$). It is straightforward to prove that Auslander-Reiten sequences in $\Ar$-mod induce Auslander-Reiten triangles in $\Ar$-\underline{mod}. In particular, we have that $\Ar$-\underline{mod} has Auslander-Reiten triangles, i.e., for all indecomposable objects $\M$ (resp. $\M''$) in $\Ar$-\underline{mod} there exists an Auslander-Reiten triangle that starts (resp. ends) at $\M$ (resp. $\M''$). However, it is not completely clear that every Auslander-Reiten triangle in $\Ar$-\underline{mod} is induced by an Auslander-Reiten sequence in $\Ar$-mod. We approach this situation in the following result.  
  
\begin{theorem}\label{thm0}
Assume that \textup{(\ref{triangle})} is an Auslander-Reiten triangle in $\Ar$-\textup{\underline{mod}}, such that neither $\M$ nor $\M'$ nor $\M''$ has projective direct summands in $\Ar$-\textup{mod}. Then there exists an Auslander-Reiten sequence in $\Ar$-\textup{mod}
\begin{equation}\label{ARseq}
0\to \M\xrightarrow{(\widehat{h},\widehat{\alpha})^t}\M'\oplus \P\xrightarrow{(\widehat{h}'',\widehat{\beta})} \M''\to 0,
\end{equation} 
such that $\P$ is a projective $\Ar$-module and \textup{(\ref{ARseq})} induces the Auslander-Reiten triangle \textup{(\ref{triangle})}. Moreover, if $\P\not=0$, then $\P$ is indecomposable, $\M\cong \mathrm{rad}\, \P$ and $\M'' \cong \P/\mathrm{soc}\, \P$ as $\Ar$-modules.  
\end{theorem}

\begin{proof}
Assume that (\ref{triangle}) is an Auslander-Reiten triangle in $\Ar$-\underline{mod} such that neither of its terms has projective direct summands in $\Ar$-mod. By \cite[Chap. I, \S 4.3]{happel}, we have that the morphisms $\underline{\h}$ and $\underline{\h}'$ are irreducible in $\Ar$-\underline{mod}. By hypothesis and Theorem \ref{giraldo0}, we obtain that $\h$ is also irreducible in $\Ar$-mod and which is not a split mono. It follows that there exists an Auslander-Reiten sequence starting at $\M$ given by
\begin{equation}\label{ARseq2}
0\to \M\xrightarrow{(\h,\widehat{\alpha})^t}\M'\oplus \Y\xrightarrow{(\h''',\widehat{\beta}')}\M'''\to 0,
\end{equation} 
which in turn induces an Auslander-Reiten triangle in $\Ar$-\underline{mod} given by
\begin{equation}\label{ARtriangle2}
\M\xrightarrow{\underline{(\h,\widehat{\alpha})}^t}\M'\oplus \Y\xrightarrow{\underline{(\h''',\widehat{\beta}')}}\M'''\xrightarrow{\underline{\widehat{w}}'}\Omega^{-1}\M.
\end{equation}
We next show that $\Y$ is a projective $\Ar$-module. Assume that $\Y=\Y'\oplus \P$, where $\Y'$ is a $\Ar$-module without projective direct summands and $\P$ is a projective $\Ar$-module. Thus we can assume further that $\widehat{\alpha}=(\widehat{\alpha}_1,\widehat{\alpha}_2)^t$, where $\widehat{\alpha}_1:\M\to \Y'$ and $\widehat{\alpha}_2:\M\to \P$ are morphisms of $\Ar$-modules. It follows by \cite[Chap. I, \S 4.5]{happel} that $\underline{(\h,\widehat{\alpha})}^t=\underline{(\h,\widehat{\alpha}_1)}^t:\M\to \M'\oplus \Y'$ is a source morphism in $\Ar$-\underline{mod}, and since $\underline{\h}$ is not split mono, there exists a morphism $\underline{\widehat{\gamma}}:\M'\oplus \Y'\to \M'$ such that $\underline{\h}=\underline{\widehat{\gamma}}\circ \underline{(\h,\widehat{\alpha}_1)}^t$. Similarly, since $\underline{\h}$ is a source morphism and $\underline{(\h,\widehat{\alpha}_1)}^t$ is not a split mono, there exists $\underline{\widehat{\psi}}: \M'\to \M'\oplus \Y'$ such that $\underline{(\h,\widehat{\alpha}_1)}^t=\underline{\widehat{\psi}}\circ \underline{\h}$. Therefore $\underline{\h}=(\underline{\widehat{\gamma}}\circ \underline{\widehat{\psi}})\circ \underline{\h}$ and $\underline{(\h,\widehat{\alpha}_1)}^t=(\underline{\widehat{\psi}}\circ \underline{\widehat{\gamma}})\circ\underline{(\h,\widehat{\alpha}_1)}^t$, which implies that $\underline{\widehat{\gamma}}\circ \underline{\widehat{\psi}}$ and $\underline{\widehat{\psi}}\circ\underline{\widehat{\gamma}}$ are both automorphisms. Hence $\M'\cong \M'\oplus \Y'$ in $\Ar$-\underline{mod}, which also implies that $\M'\cong \M'\oplus \Y'$ in $\Ar$-mod, for neither $\M'$ nor $\Y'$ has projective direct summands. This clearly implies that $\Y'=0$ and thus $\Y=\P$, which proves that $\Y$ is a projective $\Ar$-module. It follows by the axiom (T3) of triangulated categories (see e.g. \cite[Chap. I, \S 1.1]{happel}) and by e.g. \cite[Chap. I, \S 1.2]{happel} that the Auslander-Reiten triangles (\ref{triangle}) and (\ref{ARtriangle2}) are isomorphic. In particular, there is an isomorphism $\M''\cong \M'''$ in $\Ar$-\underline{mod}. Since neither $\M''$ nor $\M'''$ has projective direct summands, it follows that $\M''\cong \M'''$ are isomorphic in $\Ar$-mod, and thus we obtain that the short exact sequence (\ref{ARseq2}) induces (\ref{ARseq}). In particular, (\ref{ARseq}) is also an Auslander-Reiten sequence.

Next assume that $\P\not=0$. Since $\widehat{\alpha}:\M\to \P$ is irreducible, it follows that $\widehat{\alpha}$ is a monomorphism. Let $\widehat{\iota}_{\M}: \M\to I(\M)$ the the injective $\Ar$-hull of $\M$. Then there exists $\widehat{\delta}: I(\M)\to \P$ such that $\widehat{\alpha}=\widehat{\delta}\circ \widehat{\iota}_{\M}$. Since $\widehat{\iota}_{\M}$ is not a split mono, it follows that $\widehat{\delta}$ is a split epi. On the other hand, since $\widehat{\iota}_{\M}$ is an essential monomorphism, it follows that $\widehat{\delta}$ is also a monomorphism, which implies that $\widehat{\delta}$ is an isomorphism of $\Ar$-modules and thus $\widehat{\alpha}$ is also an essential monomorphism. Assume that $\P=\P'\oplus\P''$ with $\P'$ and $\P''$ non-zero $\Ar$-modules. Therefore, $\mathrm{Im}\, \widehat{\alpha} = (\mathrm{Im}\,\widehat{\alpha}\cap \P')\oplus (\mathrm{Im}\,\widehat{\alpha}\cap \P'')$. This shows that $\P$ is an essential extension of $\mathrm{Im} \, \widehat{\alpha}$. Therefore, $\mathrm{Im}\,\widehat{\alpha}\cap \P'\not =0$ and $\mathrm{Im}\,\widehat{\alpha}\cap \P''\not=0$. Moreover, $\M\cong \mathrm{Im}\,\widehat{\alpha}= (\mathrm{Im}\,\widehat{\alpha}\cap \P')\oplus (\mathrm{Im}\,\widehat{\alpha}\cap \P'')$, which contradict that $\M$ is indecomposable. This shows that $\P$ is indecomposable.

Now consider the natural inclusion of $\Ar$-modules  $\widehat{\iota}_{\mathrm{rad}\, \P}:\mathrm{rad}\, \P\to \P$. It follows from \cite[Example 43 (i)]{giraldo} that $\widehat{\iota}_{\mathrm{rad}\, \P}$ is irreducible, which is also right almost split. Therefore, there exists $\widehat{\lambda}: \M\to \mathrm{rad}\, \P$ such that $\widehat{\alpha}=\widehat{\iota}_{\mathrm{rad}\, \P}\circ \widehat{\lambda}$. It follows that $\widehat{\lambda}$ is split mono. Since $\mathrm{rad}\, \P$ is the unique maximal submodule of $\P$, it follows that $\widehat{\iota}_{\mathrm{rad}\, \P}$ induces a morphism of $\Ar$-modules $\widehat{\iota}_{\P}: \P/\mathrm{rad}\, \P\to \P/\mathrm{Im}\, \widehat{\lambda}$. It follows that $\widehat{\iota}_{\P}\circ \widehat{\pi}_{\P}=\widehat{\pi}_{\widehat{\alpha}}$, where $\widehat{\pi}_{\P}: \P\to \P/\mathrm{rad}\,\P$ and $\widehat{\pi}_{\widehat{\alpha}}: \P\to \P/\mathrm{Im}\,\widehat{\alpha} $ are the natural projections of $\Ar$-modules. Therefore, there exists a morphism of $\Ar$-modules $\widehat{\sigma}: \mathrm{rad}\, \P\to \M$ such that $\widehat{\iota}_{\mathrm{rad}\, \P}= \widehat{\alpha}\circ \widehat{\sigma}=\widehat{\iota}_{\mathrm{rad}\, \P}\circ(\widehat{\lambda}\circ \widehat{\sigma})$. Since $\widehat{\iota}_{\mathrm{rad}\, \P}$ is a monomorphism, it follows that $\mathrm{id}_{\mathrm{rad}\, \P}=\widehat{\lambda}\circ\widehat{\sigma}$. This implies that $\widehat{\lambda}$ is epimorphism, hence an isomorphism, i.e, $\M\cong \mathrm{rad}\, \P$ as $\Ar$-modules. The proof of the second isomorphism $\M''\cong \P/\mathrm{soc}\, \P$ is obtained by duality. This finishes the proof of Theorem \ref{thm1}.
\end{proof}
\begin{corollary}\label{cor1.8}
Consider the Auslander-Reiten sequence of $\Ar$-modules
\begin{equation}\label{ARrad}
0\to \mathrm{rad}\, \P\xrightarrow{(\h,\widehat{\iota}_{\mathrm{rad}\, \P})^t} \M'\oplus \P\xrightarrow{(\h',\widehat{\pi}_{\P})}\P/\mathrm{soc}\, \P\to 0,
\end{equation}
where $\P$ is an indecomposable projective $\Ar$-module and $\widehat{\pi}_{\P}:\P\to \P/\mathrm{soc}\, \P$ is the natural projection. Then we have the following.
\begin{enumerate}
\item The morphism $\h$ is an epimorphism, and $\h'$ is a monomorphism.
\item If $\M'=(M'_i,f'_i)_{i\in \Z}$, then there exists $i_0\in \Z$ such that $M'_i=0$ for all $i\not=i_0, i_0+1$. 
\end{enumerate}
\end{corollary}
\begin{proof}
Note that the proof of (i) is immediate from the facts that $\widehat{\iota}_{\mathrm{rad}\, \P}$ is a monomorphism and $\widehat{\pi}_{\P}$ is an epimorphism. Next we prove (ii). Note that $\P$ is of the form 
\begin{equation}
\P: \xymatrix@=20pt{
\cdots\ar@{~>}[r]&0\ar@{~>}[r]&\Hom_\A(Q,I)\ar@{~>}[r]^{\hspace*{25pt}\varphi_I}&I\ar@{~>}[r]&0\ar@{~>}[r]&\cdots,
}
\end{equation}
whose non-zero terms are in degree $i_0$ and $i_0+1$ for some $i_0\in \Z$, where $I$ is an indecomposable injective $\A$-module. Then $\mathrm{rad}\, \P$ can be viewed as 
\begin{equation}\label{radproj}
\mathrm{rad}\,\P: \xymatrix@=20pt{
\cdots\ar@{~>}[r]&0\ar@{~>}[r]&\mathrm{rad}\, \Hom_\A(Q,I)\ar@{~>}[r]^{\hspace*{30pt}\overline{\varphi}_I}&I\ar@{~>}[r]&0\ar@{~>}[r]&\cdots,
}
\end{equation}
where $\overline{\varphi}_I= \varphi_I\circ (\mathrm{id}_Q\otimes \eta)$ and $\eta: \mathrm{rad}\, \Hom_\A(Q,I)\to \Hom_\A(Q,I)$ is the natural inclusion. Therefore the morphism of $\Ar$-modules $\h: \mathrm{rad}\, \P\to \M'$ can be viewed as 
\begin{equation*}
\xymatrix@=20pt{
\cdots\ar@{~>}[r]&0\ar@{~>}[r]\ar[d]^{h_{i_0-1}}&\mathrm{rad}\, \Hom_\A(Q,I)\ar@{~>}[r]^{\hspace*{30pt}\overline{\varphi}_I}\ar[d]^{h_{i_0}}&I\ar@{~>}[r]\ar[d]^{h_{i_0+1}}&0\ar@{~>}[r]\ar[d]^{h_{i_0+2}}&\cdots\\
\cdots\ar@{~>}[r]&M'_{i_0-1}\ar@{~>}[r]^{f'_{i_0-1}}&M'_k\ar@{~>}[r]^{f'_{i_0}}&M'_{i_0+1}\ar@{~>}[r]^{f'_{i_0+1}}&M'_{i_0+2}\ar@{~>}[r]&\cdots.
}
\end{equation*}
It follows from (i) that $h$ is an epimorphism, and thus each $h_i$ is also an epimorphism for all $i\in \Z$, which implies that $M'_i=0$ for all $n\not=i_0, i_0+1$. 
\end{proof}

\section{Main results}\label{sec3}

Let $\underline{\h}:\M\to \M'$ be an irreducible morphism in $\Ar$-\underline{\textup{mod}}, where either $\M$ or $\M'$ is indecomposable and both have no projective direct summands as objects in $\Ar$-mod. As before, it follows by Theorem \ref{giraldo0} that $\h:\M\to \M'$ is also irreducible.  Thus $\h$ satisfies one of the properties (i)-(iii) in Theorem \ref{shapeirred}. Moreover, if $\h': \M\to \M'$ is another morphism in $\Ar$-mod such that $\underline{\h}=\underline{\h'}$ in $\Ar$-\underline{\textup{mod}}, then by Corollary \ref{cor0}, $\h=\h'$ in $\Ar$-mod. Thus $\h$ is smonic (resp. sepic, resp. sirreducible) if and only if so is $\h'$. This argument motivates the following definition concerning irreducible morphisms in $\Ar$-\underline{\textup{mod}}.

\begin{definition}\label{defi1.6}
Let $\underline{\h}:\M\to \M'$ be an irreducible morphism in $\Ar$-\underline{\textup{mod}}. Assume further that either $\M$ or $\M'$ is an indecomposable $\Ar$-module and that both have no projective direct summands as objects in $\Ar$-\textup{mod}. We say that $\underline{\h}$ is {\it stably smonic} (resp. {\it stably sepic}, resp. {\it stably sirreducible}) provided that $\h$ is smonic (resp. sepic, resp. sirreducible) as a morphism of $\Ar$-modules. 
\end{definition}

Our first main result is a direct consequence of Theorem \ref{giraldo0}, Corollary \ref{cor0} and Definition \ref{defi1.6}. 

\begin{theorem}\label{thm1}
Let $\underline{\h}:\M\to \M'$ be an irreducible morphism in $\Ar$-\underline{\textup{mod}}. Assume further that either $\M$ or $\M'$ is an indecomposable $\Ar$-module and that both have no projective direct summands as objects in $\Ar$-mod. Then $\underline{\h}$ is either stably smonic or stably sepic or stably sirreducible.  
\end{theorem}

Our second main result is the following which gives an interpretation of the shape of the Auslander-Reiten triangles in $\Ar$-\underline{\textup{mod}}.

\begin{theorem}\label{thm2}
Assume that (\ref{triangle}) is an Auslander-Reiten triangle, where $\M$, $\M'$ and $\M''$ have no projective direct summands as objects in $\Ar$-\textup{mod}. Then we have the following. 
\begin{enumerate}
\item If $\underline{\h}$ is stably smonic, then $\underline{\h}'$ is stably sepic.
\item If $\underline{\h}$ is stably sepic, then $\underline{\h}'$ is stably sirreducible.
\item If $\underline{\h}$ is stably irreducible, then $\underline{\h}'$ is either stably smonic or stably sirreducible.
\end{enumerate}
\end{theorem}
 
We prove Theorem \ref{thm2} by proving the following lemmata and by using the fact that every Auslander-Reiten triangle in $\Ar$-\underline{mod} is induced by an Auslander-Reiten sequence in $\Ar$-mod (see Theorem \ref{thm0}). 

\begin{lemma}\label{lem2.4}
Consider the short exact sequence of $\Ar$-modules as in (\ref{almostsplitseq}). 
\begin{enumerate}
\item If $\h$ is smonic, then $\h'$ is sepic.
\item If (\ref{almostsplitseq}) is an Auslander-Reiten sequence and $\h$ is sirreducible, then $\h'$ is sirreducible.
\end{enumerate}
\end{lemma}

\begin{proof}
(i). Assume that $\h$ is smonic.  It follows by Remark \ref{rem1.1} that for all $i\in \Z$, the short exact sequence of $\A$-modules $0\to M_i\xrightarrow{h_i}M'_i\xrightarrow{h'_i}M''_i\to 0$ splits, which implies that $h'_i$ is split epic. This proves that $\h'$ is sepic.  (ii). Assume that  (\ref{almostsplitseq}) is an Auslander-Reiten sequence and that $\h$ is a morphism sirreducible. Therefore there exists a unique $i_0\in \Z$ such that $h_{i_0}:M_{i_0}\to M'_{i_0}$ is an irreducible morphism of $\A$-modules. By using Remark \ref{rem1.1}, we can consider the following short exact sequence of $\A$-modules
\begin{equation}\label{sirreducible}
0\to M_{i_0}\xrightarrow{h_{i_0}}M'_{i_0}\xrightarrow{h'_{i_0}}M''_{i_0}\to 0.
\end{equation}
Since we are assuming that (\ref{almostsplitseq}) is an Auslander-Reiten sequence with $\h$ irreducible, it follow that $\h'$ is also irreducible. By Theorem \ref{shapeirred}, we have that $\h'$ is either smonic, or sepic, or sirreducible. If $\h'$ is smonic, then $h'_{i_0}$ is split mono and thus an isomorphism, which implies that $h_{i_0}$ is zero, which contradict that $h_{i_0}$ is irreducible. On the other hand, if $\h'$ is sepic, then $h'_{i_0}$ is split epic, which implies that (\ref{sirreducible}) splits and thus $h_{i_0}$ is split mono, which again contradicts that $h_{i_0}$ is irreducible. Therefore, $\h'$ has to be sirreducible.  
\end{proof}

\begin{lemma}\label{lemma2.5}
Consider the Auslander-Reiten sequence of $\Ar$-modules as in (\ref{ARrad}). Assume further that $\M'$ has no projective direct summands.
\begin{enumerate}
\item If $\h$ is sepic, then $\h'$ is sirreducible.
\item If $\P$ is as in (\ref{projindec}) with $\Hom_\A(Q,I)$ a simple projective $\A$-module and $\h$ is sirreducible, then $\h'$ is smonic.
\item If $\P$ is as in (\ref{projindec}) with $\Hom_\A(Q,I)$ a non-simple projective $\A$-module and $\h$ is sirreducible, then $\h'$ is sirreducible.
\end{enumerate}
\end{lemma}
\begin{proof}
(i). Assume that the non-zero terms of $\P$ are in degree $i_0$ and $i_0+1$ for some $i_0\in \Z$. It follows by Corollary \ref{cor1.8} that if $\M'=(M'_i,f'_i)_{i\in \Z}$, then $M'_i=0$ for all $i\not=i_0,i_0+1$. This implies that in this situation the Auslander-Reiten sequence (\ref{ARrad}) can be viewed as follows.

\begin{equation}\label{diagram}
\xymatrix{
	0\ar[d]:&	\cdots \ar@{~>}[r] & \ar[d]0\ar@{~>}[r] &\ar[d]0\ar@{~>}[r]& \ar[d]0\ar@{~>}[r] &\ar[d] 0\ar@{~>}[r]&\cdots \\	
\mathrm{rad}\, \P:\ar[d]^{(\h,\, \widehat{\iota}_{\mathrm{rad}\, \P})^t}	&	\cdots \ar@{~>}[r] & 0\ar[d]\ar@{~>}[r] &\mathrm{rad}\, \Hom_\A(Q,I)\ar@{~>}[r]^{\hspace*{25pt}\overline{\varphi_I}}\ar[d]^{(h_{i_0},\, \iota)^t}& I\ar[d]^{(h_{i_0+1},\, \mathrm{id}_I)^t}\ar@{~>}[r] & 0\ar[d]\ar@{~>}[r]&\cdots\\
\M'\oplus \P:\ar[d]^{(\h', \widehat{\pi}_{\P})}	&	\cdots \ar@{~>}[r] & 0\ar[d]\ar@{~>}[r] & M'_{i_0}\oplus \Hom_\A(Q,I)\ar[d]^{(h'_{i_0},\, \mathrm{id}_{ \Hom_\A(Q,I)})}\ar@{~>}[r]^{\hspace*{25pt} s_{i_0}}& M'_{i_0+1}\oplus I\ar@{~>}[r] \ar[d]^{(h'_{i_0+1},\,  \pi_{I})}& 0\ar[d]\ar@{~>}[r]&\cdots\\
	\widehat{P}/\mathrm{soc}\, \widehat{P}:\ar[d]	&	\cdots \ar@{~>}[r] & 0\ar[d]\ar@{~>}[r] & \Hom_\A(Q,I)\ar[d]\ar@{~>}[r]^{h_{i_0}}& I/\mathrm{soc}\, I\ar[d]\ar@{~>}[r] & 0\ar[d]\ar@{~>}[r]&\cdots\\	
	0&	\cdots \ar@{~>}[r] & 0\ar@{~>}[r] & 0\ar@{~>}[r]& 0\ar@{~>}[r] & 0\ar@{~>}[r]&\cdots
}
\end{equation}
where the non-zero columns are short exact sequences of $\A$-modules and $\pi_{I}:I\to I/\mathrm{soc}\, I$ is the natural projection. Therefore, $\iota=-h_{i_0}'\circ h_{i_0}$ and $\pi_{I}=-h'_{i_0+1}\circ h_{i_0+1}$. It follows that $h_{i_0}$ is a monomorphism and since $\h$ is sepic, we also have that $h_{i_0}$ is also split epic, which implies that $h_{i_0}$ is an isomorphism. Since $\iota: \mathrm{rad}\, \Hom_\A(Q,I)\to \Hom_\A(Q,I)$ is irreducible, we obtain that $h'_{i_0}$ is also irreducible in $\A$-mod. On the other hand, we also have that $h_{i_0+1}$ is split epic, which implies that $I=M'_{i_0+1}\oplus \ker h_{i_0+1}$. Suppose that $\ker h_{i_0+1}=0$, i.e. $h_{i_0+1}$ is a monomorphism. Since $\h'$ is a monomorphism of $\Ar$-modules by Corollary \ref{cor1.8} (i), it follows by Remark \ref{rem1.1} that $h'_{i_0+1}$ is also a monomorphism of $\A$-modules, which implies that $\pi_{I}$ is also a monomorphism. Therefore $\pi_{I}$ is also an isomorphism, which contradicts the fact that $\pi_{I}$ is an irreducible morphism in $\A$-mod. Hence $\ker h_{i_0+1}\not=0$. Since $I$ is indecomposable, we obtain that $M'_{i_0+1}=0$, $h_{i_0+1}=0$, $h'_{i_0+1}=0$ and thus $\pi_{I}=0$.  Thus $\h'$ is sirreducible. 

(ii). As before, we associate to the Auslander-Reiten sequence (\ref{ARrad}) a diagram as in (\ref{diagram}). Assume next that $\h$ is sirreducible and that the term $\Hom_\A(Q, I)$ in $\P$ is a simple projective $\A$-module.  This implies that $\mathrm{rad}\, \Hom_\A(Q, I)=0$. By Corollary \ref{cor1.8} (i), we have that $\h$ is un epimorphism, which implies that $h_{i_0}=0$, $M'_{i_0}=0$ and $h'_{i_0}=0$ which is trivially a split mono. Since $\h$ is sirreducible, we have that $h_{i_0+1}$ is an irreducible morphism of $\A$-modules which together with the fact that $\pi_{I}$ is also irreducible, we obtain that $h'_{i_0+1}$ is a split epi morphism of $\A$-modules. On the other hand, again by Corollary \ref{cor1.8} (i) we obtain that $\h'$ is a monomorphism of $\Ar$-modules, which together with Remark {rem1.1} gives us that $h'_{i_0+1}$ is an isomorphism, which is trivially a split mono. Thus $\h'$ is smonic. 

(iii). As before, we can consider the diagram (\ref{diagram}), where $h'_{i_0}$ and $h'_{i_0+1}$ are both monomorphisms of $\A$-modules. In particular, we have that $h'_{i_0+1}$ is an isomorphism. Since $\h$ is sirreducible, it follows that either $h_{i_0}$ is irreducible and $h_{i_0+1}$ is a split mono, or $h_{i_0}$ is split epi and $h_{i_0+1}$ is irreducible morphism of $\A$-modules. If the first case holds, then together with the fact that $\iota: \mathrm{rad}\, \Hom_\A(Q,I)\to \Hom_\A(Q,I)$ is irreducible, we obtain that $h'_{i_0}$ is split epi, which further implies that $h'_{i_0}$ is an isomorphism. This proves that $\h'$ is an isomorphism which contradicts that $\h'$ is irreducible morphism of $\Ar$-modules. It follows that $h_{i_0}$ is split epi and $h_{i_0+1}$ is irreducible morphism of $\A$-modules. Since $h'_{i_0}$ cannot be split epi for this will give again the above contradiction, it follows that $h_{i_0}$ is also split mono, which implies that $h_{i_0}$ is an isomorphism. This gives us that $h'_{i_0}$ is irreducible and hence $\h'$ is sirreducible. This finishes the proof of Lemma \ref{lemma2.5}.
\end{proof}

\begin{remark}
Note that in the situation of Lemma \ref{lemma2.5} (i), we obtain that if $\P$ is as in (\ref{projindec}), then $I$ is a simple injective $\A$-module. 
\end{remark}

\begin{proof}[Proof of Theorem \ref{thm2}]
Assume that (\ref{triangle}) is an Auslander-Reiten triangle, where $\M$, $\M'$ and $\M''$ have no projective direct summands as objects in $\Ar$-\textup{mod}. Then by Theorem \ref{thm0}, there exists an Auslander-Reiten sequence as in (\ref{ARseq}) that induces (\ref{triangle}) and such that $\underline{\h}'$ is equal to $\underline{\h}''$ module an isomorphism. Thus by Corollary \ref{cor0}, we can assume that $\h'=\h''$. If $\underline{\h}$ is stably smonic, then by definition $\h$ is smonic, which implies that $\h$ is also a monomorphism. Assume that $\P\not=0$ in (\ref{ARseq}). Then by Theorem \ref{thm0}, it follows that $\P$ is indecomposable projective $\Ar$-module and the Auslander-Reiten sequence (\ref{ARseq}) is as in (\ref{ARrad}), which together with Corollary \ref{cor1.8} (i) gives that $\h$ is also an epimorphism. This contradicts the fact that $\h$ is irreducible, and thus $\P=0$. Then by Lemma  \ref{lem2.4} (i) we obtain that $\underline{\h}'$ is sepic and thus $\underline{\h}'$ is stably sepic. Next assume that $\underline{\h}$ is stably sepic. In this situation, if $\P=0$ in (\ref{ARseq}), then we obtain that $\h$ is an isomorphism which again contradicts that $\h$ is irreducible. Thus $\P\not=0$ and the Auslander-Reiten sequence (\ref{ARseq}) becomes as in (\ref{ARrad}).  Then by Lemma \ref{lemma2.5} (i), we obtain that $\h'$ is sirreducible, which gives that $\underline{\h}'$ is stably sirreducible. Finally assume that $\underline{\h}$ is stably sirreducible. If $\P=0$, then the Auslander-Reiten sequence (\ref{ARseq}) is as in (\ref{almostsplitseq}). By Lemma \ref{lem2.4} (ii), we obtain that $\h'$ is sirreducible. If $\P\not=0$, then as before (\ref{ARseq}) is as in (\ref{ARrad}) and $\P$ is an indecomposable projective $\Ar$-module, i.e. $\P$ is of the form (\ref{projindec}). If the term $\Hom_\A(Q,I)$ of $\P$ is simple, then we obtain by Lemma \ref{lemma2.5} (ii) that $\h'$ is smonic and thus $\underline{\h}'$ is stably smonic. Similarly, if $\Hom_\A(Q,I)$ is non-simple, then by Lemma \ref{lemma2.5} (iii) we obtain that $\h'$ is sirreducible and thus $\underline{\h}'$ is stably sirreducible. This finishes the proof of Theorem \ref{thm2}.

\end{proof}

\section{An example}\label{sec4}
Let $\A=\k Q/\langle \rho\rangle$ be the basic finite dimensional algebra whose quiver with relations is given as follows:
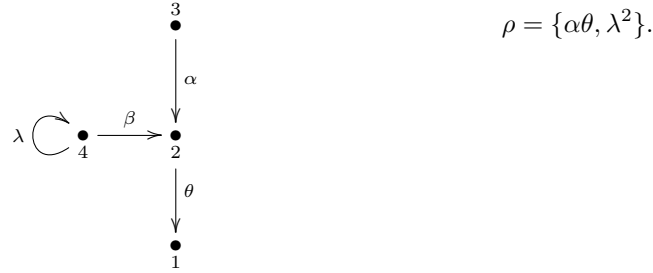
\begin{align*}
\xymatrix{&\overset{3}{\bullet}\ar[d]^{\alpha}\\
		\underset{4}{\bullet} \ar@(dl,ul)[]^{\large \lambda}\ar[r]^{\beta} & \underset{2}{\bullet}\ar[d]^{\theta}\\
	&\underset{1}{\bullet}}&& \rho=\{\alpha\theta, \lambda^2\}.
\end{align*}
Observe that $\A$ is a gentle algebra in the sense of \cite{assem} of infinite global dimension. It follows from the results in \cite{schroer} that $\Ar$ is special biserial (in the sense of \cite{wald}) and that $\Ar=\k \widehat{Q}/\langle \widehat{\rho}\rangle$, where $\widehat{Q}$ is given by 
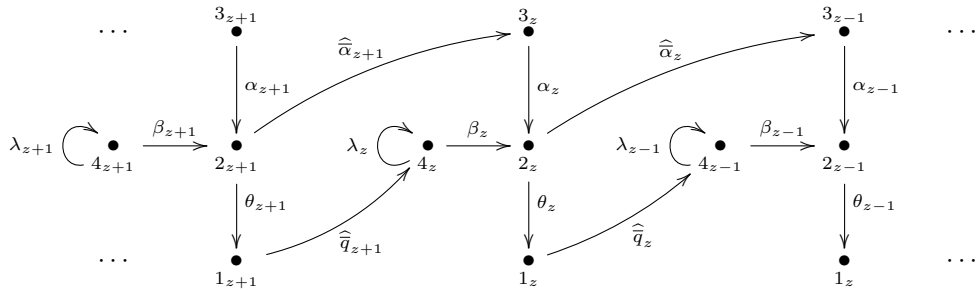
\begin{equation*}
\xymatrix{\cdots&\overset{3_{z+1}}{\bullet}\ar[d]^{\alpha_{z+1}}&&&\overset{3_z}{\bullet}\ar[d]^{\alpha_z}&&&\overset{3_{z-1}}{\bullet}\ar[d]^{\alpha_{z-1}}&\cdots\\
	\underset{4_{z+1}}{\bullet} \ar@(dl,ul)[]^{\large \lambda_{z+1}}\ar[r]^{\beta_{z+1}} & \underset{2_{z+1}}{\bullet}\ar[d]^{\theta_{z+1}}\ar@/^{3mm}/[urrr]^{\widehat{\overline{\alpha}}_{z+1}}&&\underset{4_z}{\bullet} \ar@(dl,ul)[]^{\large \lambda_z}\ar[r]^{\beta_z} &\underset{2_z}{\bullet}\ar[d]^{\theta_z}\ar@/^{3mm}/[urrr]^{\widehat{\overline{\alpha}}_z}&&\underset{4_{z-1}}{\bullet} \ar@(dl,ul)[]^{\large \lambda_{z-1}}\ar[r]^{\beta_{z-1}} &\underset{2_{z-1}}{\bullet}\ar[d]^{\theta_{z-1}}\\
	\cdots&\underset{1_{z+1}}{\bullet}\ar@/_{3mm}/[urr]_{\widehat{\overline{q}}_{z+1}}&&&\underset{1_z}{\bullet}\ar@/_{2mm}/[urr]_{\widehat{\overline{q}}_{z}}&&&\underset{1_{z}}{\bullet}&\cdots}
\end{equation*}
and $\widehat{\rho}$ is given by 
\begin{equation*}
\left\{
\begin{array}{c|c}
\alpha_z\theta_z,\lambda_z^2, \beta_z\widehat{\overline{\alpha}}_z, \widehat{\overline{q}}_z\beta_{z-1}\alpha_z\widehat{\overline{\alpha}}_z \alpha_{z-1},\widehat{\overline{\alpha}}_z \alpha_{z-1}\widehat{\overline{\alpha}}_{z-1},&\\
\lambda_z\beta_z\theta_z\widehat{\overline{q}}_z\lambda_{z-1},\beta_z\theta_z\widehat{\overline{q}}_z\lambda_{z-1}\beta_{z-1},\theta_z\widehat{\overline{q}}_z\lambda_{z-1}\beta_{z-1}\theta_{z-1},\widehat{\overline{q}}_z\lambda_{z-1}\beta_{z-1}\theta_{z-1}\widehat{\overline{q}}_{z-1},& z\in \Z\\ 
\beta_z\theta_z\widehat{\overline{q}}_z\lambda_{z-1}-\lambda_z\beta_z\theta_z\widehat{\overline{q}}_z, \widehat{\overline{\alpha}}_z \alpha_{z-1}- \theta_z\widehat{\overline{q}}_z\lambda_{z-1}\beta_{z-1}&
\end{array}
\right\}.
\end{equation*}
It follows that the radical series of the indecomposable projective $\Ar$-modules $\widehat{Q}$ can be represented as follows. 

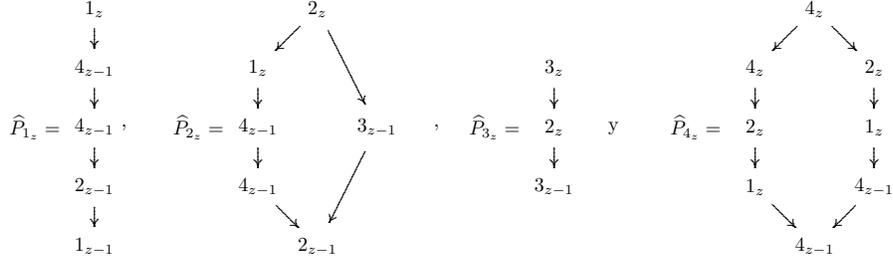
\begin{equation*}
	\begindc{\commdiag}[300]
	\obj(-5,0)[p1]{$	\scalebox{0.75}{\begindc{\commdiag}[150]
		\obj(0,0)[p1]{$1_z$}
		\obj(0,-2)[p2]{$4_{z-1}$}
		\obj(0,-4)[p3]{$4_{z-1}$}
		\obj(0,-6)[p4]{$2_{z-1}$}
		\obj(0,-8)[p5]{$1_{z-1}$}
		\obj(-2,-4)[p6]{$\P_{1_{z}}=$}
		\obj(1,-4)[p7]{$,$}
		\mor{p1}{p2}{}
		\mor{p2}{p3}{}
		\mor{p3}{p4}{}
		\mor{p4}{p5}{}
		\enddc}$}
	\obj(-2,0)[p2]{$\scalebox{0.75}{\begindc{\commdiag}[150]
		\obj(0,0)[p1]{$2_z$}
		\obj(-2,-2)[p2]{$1_{z}$}
		\obj(-2,-4)[p3]{$4_{z-1}$}
		\obj(-2,-6)[p4]{$4_{z-1}$}
		\obj(0,-8)[p5]{$2_{z-1}$}
		\obj(2,-4)[p6]{$3_{z-1}$}
		\obj(-4,-4)[p7]{$\P_{2_z}=$}	
		\obj(4,-4)[p8]{$,$}	
		\mor{p1}{p2}{}
		\mor{p2}{p3}{}
		\mor{p3}{p4}{}
		\mor{p4}{p5}{}
		\mor{p1}{p6}{}
		\mor{p6}{p5}{}
		\enddc}$}
	\obj(1,0)[p3]{$\scalebox{0.75}{\begindc{\commdiag}[150]
				\obj(0,0)[p1]{$3_z$}
		\obj(0,-2)[p2]{$2_z$}
		\obj(0,-4)[p3]{$3_{z-1}$}
		\obj(-2,-2)[p4]{$\P_{3_z}=$}
			\obj(2,-2)[p5]{$\mbox{y}$}
			\mor{p1}{p2}{}
			\mor{p2}{p3}{}			
		\enddc}$}
	
	\obj(4,0)[p5]{$\scalebox{0.75}{\begindc{\commdiag}[150]
		\obj(0,0)[p1]{$4_z$}
		\obj(-2,-2)[p2]{$4_z$}
		\obj(-2,-4)[p3]{$2_z$}
		\obj(-2,-6)[p4]{$1_z$}
		\obj(0,-8)[p5]{$4_{z-1}$}
		\obj(2,-2)[p6]{$2_z$}
		\obj(2,-4)[p7]{$1_z$}
		\obj(2,-6)[p8]{$4_{z-1}$}
		\obj(-4,-4)[p9]{$\P_{4_z}=$}
		\mor{p1}{p2}{}
		\mor{p2}{p3}{}
		\mor{p3}{p4}{}
		\mor{p4}{p5}{}
		\mor{p1}{p6}{}
		\mor{p6}{p7}{}
		\mor{p7}{p8}{}
		\mor{p8}{p5}{}
		\enddc}$}
	\enddc
	\end{equation*}
Since $\A$ and $\Ar$ are both special biserial, it follows that all the indecomposable non-projective $\A$-modules and $\Ar$-modules, respectively can be represented combinatorially by using so-called strings and bands for $\A$ and $\Ar$, respectively. The corresponding indecomposable modules are called string and band $\A$-modules and $\Ar$-modules. respectively. Moreover, the irreducible morphisms between string $\Ar$-modules are completely described by using so-called hooks and co-hooks. If $S$ is a string for $\A$, we denote by $M[S]$ the corresponding string $\A$-module. In particular, if $v$ is a vertex of the quiver of $\A$, we denote by $M[\1_v]$ the corresponding simple $\A$-module. On the other hand, if $\widehat{S}$ is a string for $\Ar$, we denote by $\M[\widehat{S}]$ the corresponding string $\Ar$-module, and if $\hat{v}$ is a vertex of the quiver of $\Ar$, then $\M[\1_{\hat{v}}]$ denotes the corresponding simple $\Ar$-module. For more details and definitions of strings and band modules see \cite{buri}. 

It follows that the component of the stable Auslander-Reiten quiver of $\Ar$ that contains the simple $\Ar$-modules corresponding to the vertices $1_z$, $2_z$ and $3_z$ looks like as in Figure \ref{fig1}. Note that this component is of type $\Z\mathbb{A}_\infty$.

\begin{figure}[ht]
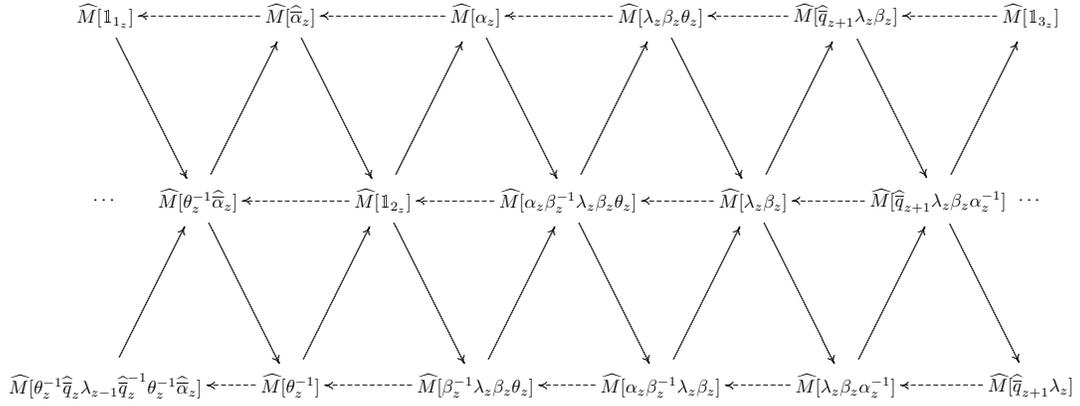

\centering
$$
		\scalebox{0.70}{
		
	\begindc{\commdiag}[250]
	\obj(-12,4)[p1]{$\M[\1_{1_z}]$}
	\obj(-8,4)[p2]{$\M[\widehat{\overline{\alpha}}_z]$}
	\obj(-4,4)[p3]{$\M[\alpha_z]$}	
	\obj(0,4)[p4]{$\M[\lambda_z\beta_z\theta_z]$}		
	\obj(4,4)[p5]{$\M[\widehat{\overline{q}}_{z+1}\lambda_z\beta_z]$}
	\obj(8,4)[p6]{$\M[\1_{3_z}]$}
	\obj(-10,0)[p7]{$\M[\theta_z^{-1}\widehat{\overline{\alpha}}_z]$}
	\obj(-6,0)[p8]{$\M[\1_{2_z}]$}
	\obj(-2,0)[p9]{$\M[\alpha_z\beta_z^{-1}\lambda_z\beta_z\theta_z]$}	
	\obj(2,0)[p10]{$\M[\lambda_z\beta_z]$}		
	\obj(6,0)[p11]{$\M[\widehat{\overline{q}}_{z+1}\lambda_z\beta_z\alpha_z^{-1}]$}
	\obj(-12,-4)[p12]{$\M[\theta_z^{-1}\widehat{\overline{q}}_z\lambda_{z-1}\widehat{\overline{q}}_z^{-1}\theta_z^{-1}\widehat{\overline{\alpha}}_z]$}
	\obj(-8,-4)[p13]{$\M[\theta_z^{-1}]$}
	\obj(-4,-4)[p14]{$\M[\beta_z^{-1}\lambda_z\beta_z\theta_z]$}		
	\obj(0,-4)[p15]{$\M[\alpha_z\beta_z^{-1}\lambda_z\beta_z]$}	
	\obj(4,-4)[p16]{$\M[\lambda_z\beta_z\alpha_z^{-1}]$}
	\obj(8,-4)[p17]{$\M[\widehat{\overline{q}}_{z+1}\lambda_z]$}
 \obj(8,0)[p19]{$\cdots$}
	\obj(-12,0)[p18]{$\cdots$}
		\mor{p1}{p7}{}
		\mor{p7}{p2}{}
		\mor{p7}{p13}{}
		\mor{p2}{p8}{}
		\mor{p8}{p3}{}
		\mor{p8}{p7}{}[+1,1]
		\mor{p8}{p14}{}
		\mor{p9}{p4}{}
		\mor{p10}{p5}{}
		\mor{p12}{p7}{}
		\mor{p11}{p6}{}
		\mor{p13}{p8}{}
		\mor{p14}{p9}{}
		\mor{p15}{p10}{}
		\mor{p16}{p11}{}
		\mor{p9}{p15}{}
		\mor{p10}{p16}{}
		\mor{p11}{p17}{}
		\mor{p3}{p9}{}
		\mor{p4}{p10}{}
		\mor{p5}{p11}{}
     	\mor{p2}{p1}{}[+1,1]
	   \mor{p3}{p2}{}[+1,1]
	   \mor{p4}{p3}{}[+1,1]
	   \mor{p5}{p4}{}[+1,1]
	   \mor{p6}{p5}{}[+1,1]
	   \mor{p9}{p8}{}[+1,1]
	   \mor{p10}{p9}{}[+1,1]
	   \mor{p11}{p10}{}[+1,1]
	   \mor{p13}{p12}{}[+1,1]
	   \mor{p14}{p13}{}[+1,1]
	   \mor{p15}{p14}{}[+1,1]
	   \mor{p16}{p15}{}[+1,1]
	   \mor{p17}{p16}{}[+1,1]
	\enddc}
	$$	
\caption{The component of the stable Auslander-Reiten quiver of $\Ar$ containing the simple $\Ar$-modules corresponding to the vertices $1_z$, $2_z$ and $3_z$ with $z\in \Z$.}\label{fig1}
\end{figure}

In the following, we check that $\Ar$ verifies all the possibilities in Theorem \ref{thm2}.    In the followsing $z\in\Z$ is a fixed both arbitrary integer. 
\begin{enumerate}

\item Consider the Auslander-Reiten triangle 
\begin{equation}\label{exam1}
\M[\theta_z^{-1}\widehat{\overline{\alpha}}_z]\xrightarrow{\underline{\h}}\M[\widehat{\overline{\alpha}}]\oplus \M[\theta_z^{-1}]\xrightarrow{\underline{\h}'}\M[\1_{2_z}]\xrightarrow{\underline{\h}''}\Omega^{-1}\M[\theta_z^{-1}\widehat{\overline{\alpha}}]. 
\end{equation}
Then (\ref{exam1}) can be represented as follows:
$$\xymatrix{
\M[\theta_z^{-1}\widehat{\overline{\alpha}}_z]:\ar[d]_(.4){\underline{\h}}^(0.4){}&\cdots\ar@{~>}[r]&0\ar[d]\ar@{~>}[r]&M[\theta]\ar[d]\ar@{~>}[r]&M[\1_3]\ar[d]^{\mathrm{id}_{M[\1_3]}}\ar@{~>}[r]&\cdots\\
\M[\widehat{\overline{\alpha}}]\oplus \M[\theta_z^{-1}]:\ar[d]_(.4){\underline{\h}'}^(0.4){}&\cdots\ar@{~>}[r]&0\ar[d]\ar@{~>}[r]&M[\theta]\oplus M[\1_2]\ar@{~>}[r]\ar[d]&M[\1_{3}]\ar[d]\ar@{~>}[r]&\cdots\\
				\M[\1_{2_z}]\:\ar[d]_(.5){\underline{\h}''}&\cdots\ar@{~>}[r]&0\ar[d]\ar@{~>}[r]&M[\1_2]\ar[d]\ar@{~>}[r]&0\ar[d]\ar@{~>}[r]&\cdots\\
				\Omega^{-1}\M[\theta_z^{-1}\widehat{\overline{\alpha}}_z]:&\cdots\ar@{~>}[r]&M[\1_1]\ar@{~>}[r]&M[\lambda\beta\alpha^{-1}]\ar@{~>}[r]&0\ar@{~>}[r]&\cdots}$$ 
where the non-zero columns are in degrees $z-1$, $z$ and $z+1$. Note that $\underline{\h}$ is smonic and $\underline{\h}'$ is sepic. This verifies Theorem \ref{thm2} (i).  

\item Consider the Auslander-Reiten triangle 
\begin{equation}\label{exam2}
 \M[\widehat{\overline{\alpha}}_z]\xrightarrow{\underline{\h}}\M[\1_{2_z}]\xrightarrow{\underline{\h}'}\M[\alpha_z]\xrightarrow{\underline{\h}''}\Omega^{-1}\M[\widehat{\overline{\alpha}}_z]. 
\end{equation}
Then (\ref{exam2}) can be represented as follows
$$\xymatrix{\M[\widehat{\overline{\alpha}}]:\ar[d]_(.4){\underline{\h}}^(0.4){}&\cdots\ar@{~>}[r]&M[\1_2]\ar[d]^{\mathrm{id}_{M[\1_2]}}\ar@{~>}[r]&M[\1_3]\ar[d]\ar@{~>}[r]&\cdots\\
				\M[\1_{2_z}]:\ar[d]_(.4){\underline{\h}'}^(0.4){}&\cdots\ar@{~>}[r]&M[\1_2]\ar@{~>}[r]\ar[d]&0\ar[d]\ar@{~>}[r]&\cdots\\
				\M[\alpha_z]:\ar[d]_(.5){\underline{\h}''}&\cdots\ar@{~>}[r]&M[\alpha]\ar[d]\ar@{~>}[r]&0\ar[d]\ar@{~>}[r]&\cdots\\
				\Omega^{-1}\M[\widehat{\overline{\alpha}}]:&\cdots\ar@{~>}[r]&M[\1_3]\ar@{~>}[r]&0\ar@{~>}[r]&\cdots}$$ 
where the non-zero columns are in degrees $z$ and $z+1$. Note that the morphism $M[\1_2]\to M[\alpha]$ is irreducible and thus $\underline{\h}'$ is stably sirreducible. Note also that $\underline{\h}$ is stably sepic. This verifies Theorem \ref{thm2} (ii).

\item  Consider next the Auslander-Reiten triangle
\begin{equation}\label{exam3}
 \M[\lambda_z\beta_z\theta_z]\xrightarrow{\underline{\h}}\M[\lambda_z\beta_z]\xrightarrow{\underline{\h}'}\M[\widehat{\overline{q}}_{z+1}\lambda_z\beta_z]\xrightarrow{\underline{\h}''}\Omega^{-1} \M[\lambda_z\beta_z\theta_z]
 \end{equation}
Then (\ref{exam3}) can be represented as follows:

$$\xymatrix{\M[\lambda_z\beta_z\theta_z]:\ar[d]_(.4){\underline{\h}}^(0.4){}&\cdots\ar@{~>}[r]&0\ar[d]\ar@{~>}[r]&M[\lambda\beta\theta]\ar[d]^{}\ar@{~>}[r]&\cdots\\
			\M[\lambda_z\beta_z]:\ar[d]_(.4){\underline{\h}'}^(0.4){}&\cdots\ar@{~>}[r]& 0\ar@{~>}[r]\ar[d]^{}&M[\lambda\beta]\ar[d]^{\mathrm{id}_{M[\lambda\beta]}}\ar@{~>}[r]&\cdots\\
			\M[\widehat{\overline{q}}_{z+1}\lambda_z\beta_z]:\ar[d]_(.5){\underline{\h}''}&\cdots\ar@{~>}[r]&M[\1_1]\ar[d]\ar@{~>}[r]&M[\lambda\beta]\ar[d]\ar@{~>}[r]&\cdots\\
			\Omega^{-1} \M[\lambda_z\beta_z\theta_z]:&\cdots\ar@{~>}[r]&M[\1_1]\ar@{~>}[r]&0\ar@{~>}[r]&\cdots}$$

where the non-zero columns are in degrees $z$ and $z+1$. Note that the morphism $M[\lambda\beta\theta]\to M[\lambda\beta]$ is irreducible and thus $\underline{\h}$ is stably sirreducible. Note also that $\underline{\h}'$ is stably smonic. This verifies the first situation of Theorem \ref{thm2} (iii).  

\item Finally, consider the Auslander-Reiten triangle 
\begin{equation}\label{exam4}
\M[\1_{1_z}]\xrightarrow{\underline{\h}}\M[\theta_z^{-1}\widehat{\overline{\alpha}}_z]\xrightarrow{\underline{\h}'}\M[\widehat{\overline{\alpha}}_z]\xrightarrow{\underline{\h}''}\Omega^{-1}\M[\1_{1_z}]. 
\end{equation}
Then (\ref{exam4}) can be represented as follows:
$$\xymatrix{\M[\1_{1_z}]:\ar[d]_(.4){\underline{\h}}^(0.4){}&\cdots\ar@{~>}[r]&0\ar[d]\ar@{~>}[r]&M[\1_1]\ar[d]\ar@{~>}[r]&0\ar[d]\ar@{~>}[r]&\cdots\\
				\M[\theta_z^{-1}\widehat{\overline{\alpha}}_z]:\ar[d]_(.4){\underline{\h}'}^(0.4){}&\cdots\ar@{~>}[r]&0\ar[d]\ar@{~>}[r]&M[\theta^{-1}]\ar@{~>}[r]\ar[d]&M[\1_3]\ar[d]^{\mathrm{id}_{M[\1_3]}}\ar@{~>}[r]&\cdots\\
				\M[\widehat{\overline{\alpha}}_z]:\ar[d]_(.5){\underline{\h}''}&\cdots\ar@{~>}[r]&0\ar[d]\ar@{~>}[r]&M[\1_2]\ar[d]\ar@{~>}[r]&M[\1_3]\ar[d]\ar@{~>}[r]&\cdots\\
				\Omega^{-1} \, \M[\1_{1_z}]:&\cdots\ar@{~>}[r]&M[\1_1]\ar@{~>}[r]&M[(\lambda\beta)^{-1}]\ar@{~>}[r]&0\ar@{~>}[r]&\cdots}$$ 
where the non-zero columns are in degrees $z-1$, $z$ and $z+1$. Note that the morphisms $M[\1_1]\to M[\theta^{-1}]$ and $M[\theta^{-1}]\to M[\1_2]$ are both irreducible in $\A$-mod. Thus $\underline{\h}$ and $\underline{\h}'$ are both stably sirreducible. This verifies the second situation of Theorem \ref{thm2} (iii).  

\end{enumerate}
\subsection*{Acknoledgments}
All the authors would like to express their gratitude to Professor Raymundo Bautista for providing the main ideas used to prove Theorem \ref{thm0} and for thoughtful comments and suggestion regarding this research, during the visit of the first and the second author to the Universidad Nacional Aut\'onoma de M\'exico in Morelia during March 2019. Part of this research was also performed during the visit of the first author to third one at the Valdosta State University during Spring 2018.   
\bibliographystyle{amsplain}
\bibliography{Auslander-Reiten_Triangles}

\end{document}